\documentclass[reqno, 12pt]{amsart}
\def\C{{\mathbb C}}

\def\Q{{\mathbb Q}}
\def\Qp{{\mathbb Q}_p}
\def\Z{{\mathbb Z}}
\def\Zp{{\mathbb Z}_p}

\def\ord{\operatorname{ord}}

\def\F{{\mathbb F}}
\def\Qp {{{\mathbb Q}_p}}
\def\Fq{{{\mathbb F}_q}}

\newtheorem{lemma}{Lemma}
\newtheorem{cor}{Corollary}
\newtheorem{prop}{Proposition}

\newtheorem{theorem}{Theorem}

\theoremstyle{definition}
\newtheorem{defn}{Definition}   

\newtheorem{question}{Question}

\theoremstyle{remark}
\newtheorem{rem}{Remark}        
      
\newtheorem{example}{Example}

\begin{document}
\title[$\zeta$-phenomenology]
{$\zeta$-phenomenology}
\author{David Goss}
\thanks{}
\address{Department of Mathematics\\The Ohio State University\\231
W.\ $18^{\rm th}$ Ave.\\Columbus, Ohio 43210}

\email{goss@math.ohio-state.edu}
\thanks{This paper is dedicated to  my wife Rita.}
\date{Winter, 2010}

\begin{abstract}
It is well known that Euler experimentally discovered the functional equation
of the Riemann zeta function. Indeed he detected the fundamental
$s\mapsto 1-s$ invariance of $\zeta(s)$ by looking only at special values.
In particular, via this functional equation, the permutation group on two letters, $S_2\simeq\Z/(2)$, is realized 
as a group of symmetries of $\zeta(s)$. In this paper, we use the theory of
special-values of our characteristic $p$ zeta functions to experimentally
detect a natural symmetry group $S_{(q)}$ for these functions
of cardinality ${\mathfrak c}=2^{\aleph_0}$ (where $\mathfrak c$
is the cardinality of the continuum); $S_{(q)}$ is a realization of the permutation 
group on $\{0,1,2\ldots\}$ as homeomorphisms
of $\Zp$ stabilizing both the nonpositive and nonnegative integers.  
We present a number of distinct instances in which $S_{(q)}$ acts (or appears to
act) as symmetries of our functions.  In particular, we 
present a natural, but highly mysterious, action of $S_{(q)}$ on a large subset of the
domain of our functions that appears to stabilize zeta-zeroes. 
As of this writing, we do not
yet know an overarching formalism that unifies these examples; however,
it would seem that this formalism will involve an interplay between the $1$-unit group
$U_1$ -- playing the role of a ``gauge group'' -- and $S_{(q)}$.
Furthermore, we show that $S_{(q)}$ may be naturally realized as an automorphism 
group of the convolution algebras of characteristic $p$ valued measures.
\end{abstract}

\maketitle

\section{Introduction}\label{intro}
Euler's work on the zeta function has been an inspiration to us for
many years. This work is briefly summarized in our Section \ref{discovery},
but we highly recommend \cite{ay1} to the reader. Euler was able to compute
the values of the Riemann zeta function at the positive even integers
and at the negative integers. By very cleverly comparing them, he found
the basic symmetry given by the famous functional equation of $\zeta(s)$.
In particular, the lesson Euler taught us was that the special values are
a window allowing one to glimpse very deep internal structure of the zeta
function.

In the characteristic $p$ theory, we have long had good results on
special values in the basic case where the base ring $A$ is
$\Fq[T]$. At the positive integers, one had the brilliant analog
of Euler's results due to L.Carlitz in the 1930's and 1940's
(complete with an analogs of
$2\pi i$, Bernoulli numbers, factorials, etc.). At the negative integers,
one also had good formulas for the values of the characteristic
$p$ zeta function. However, all attempts to put the positive and negative
integers together in an ``$s\mapsto 1-s$'' fashion failed.

The theory of these characteristic $p$ functions works for any of 
Drinfeld's rings $A$ (the affine ring of a complete smooth curve over
$\Fq$ minus a fixed closed point $\infty$). It is however substantially
harder to do explicit calculations for general $A$ and so there are not yet
many specific examples to study.

In the 1990's Dinesh Thakur \cite{th1} looked at the ``trivial-zeroes'' of these
zeta-functions for certain nonpolynomial rings $A$. He found the intriguing
phenomenon that such trivial-zeroes may have a higher order of vanishing
than naturally arises from the theory (current theory only gives a very 
classical looking 
lower bound on this vanishing in general, not the exact order!). More recently 
Javier Diaz-Vargas \cite{dv2} extended Thakur's calculations. Both Thakur and
Diaz-Vargas experimentally found that this general higher vanishing at $-j$
appears to be associated with $j$ of a very curious type: the sum of
the $q$-adic digits of $j$ must be bounded. 

We call a trivial zero ``regular'' if its order is exactly the classical
lower bound and ``irregular'' otherwise. (See Subsection \ref{classbern}
for a discussion of this terminology.)

We have been trying to come to grips with the implications of Thakur's
and Diaz-Vargas's inspired calculations for a few years; see \cite{go2}. 
Recently we discovered a huge group of symmetries that seems to 
underlie their calculations. It is our purpose here to describe this group
and its relationship to their calculations as well as other instances
when it arises. In particular, we shall see how this group conjecturally
allows one to establish certain {\it finiteness} results on trivial
zeroes for characteristic $p$ zeta-functions. Moreover, calculations
in the polynomial ring case also indicate that this group acts on the
zeroes of the zeta-function (see Subsection \ref{assocomp}).

This paper is written to quickly explain these implications to the reader.
We try as much as possible to stay away from general theory and keep
the paper as self-contained as possible. 

The symmetry group $S_{(q)}$ is introduced in Section \ref{sp}. It is
a group of homeomorphisms of $\Zp$ obtained by simply rearranging the
$q$-expansion coefficients. 
In particular, we readily see that
$S_{(q)}$ stabilizes {\it both} the nonpositive and nonnegative integers; 
there is no mixing as in $s\mapsto 1-s$. Thus,
perhaps, we have the ``true''
explanation for  the failure to somehow put the positive and negative integers
together as Euler did.  We will see in Section \ref{negev} that the action
of $S_{(q)}$ appears to preserve the orders of trivial zeroes coming
from Thakur's and Diaz-Vargas' calculations (which is how
we actually discovered it).

Moreover, all may not be lost here in terms of relating the positive and
negative integers. Indeed, we shall see in Section \ref{posev} that 
 $S_{(q)}$ may also be realized as symmetries of
Carlitz's ``von Staudt-Clausen'' result where he calculated the denominator
of his Bernoulli analogs at the positive integers in the basic $\Fq[T]$-case.
This is very exciting and highly mysterious to us.

Remarkably, we further establish for $A=\Fq[T]$
that $S_{(q)}$ determines the degree of the ``special polynomials'' (see Definition
\ref{special})
that arise in the theory; i.e., these degrees are an {\it invariant} for the
action of $S_{(q)}$. Since the trivial zeroes are zeroes of the special
polynomials, knowledge of the degrees of these polynomials would obviously
bound the orders of trivial zeroes.

As is well-known by now, these characteristic $p$ zeta-functions are 
analytically continued by ``summing according to degree''. For each fixed
degree $d$, one obtains continuous functions from $\Zp$ into our finite
characteristic fields. We further show that the zero sets of these functions,
as subsets of $\Zp$, 
are {\it stable} under the action of $S_{(q)}$, and we present some evidence
that the same result holds for the $v$-adic interpolations of $\zeta(s)$ at
the finite primes $v$. 

We shall frequently use the essential paper of J. Sheats 
\cite{sh1} which, in particular,
established rigorously some results mentioned much
earlier by L. Carlitz \cite{ca4}. Sheats did this as he proved
the ``Riemann hypothesis for $\Fq[T]$''. In other words, our functions
are naturally defined on the space $\mathbb S_\infty$ (See Definition
\ref{inf5}) which is a product of a characteristic $p$ variable $x$
and  $p$-adic variable $y$; Sheats then shows that, upon fixing $y$, 
all zeroes in $x^{-1}$ 
are simple and uniquely determined by their absolute values.
So, Galois invariance immediately implies that the zeroes must lie
on the line $\Fq((1/T))$.

Note that the classical functional equation, $s\mapsto 1-s$ becomes
$t\mapsto -t$ upon setting $s=1/2+it$ and obviously 
$\vert t\vert=\vert -t\vert$.

In the function field situation, our whole theory begins by choosing
a notion of ``positive'' (generalizing the notion of ``monic polynomial'')
and then fixing a positive parameter $\pi$. So for instance if we use
the standard notion of monic, one could choose $\pi=1/T$. Once
we have chosen $\pi$, we are then able to ``exponentiate ideals'' and define
$L$-series. Passing from one positive parameter $\pi_1$ to another
$\pi_2$ multiplies the $x$-coordinate of a zeta-zero 
with a certain power of the ``phase'' $u=\pi_1/\pi_2$ (which is a $1$-unit,
see Lemma \ref{depend1}).
Clearly multiplication by a power of $u$ does not change the absolute value of
the $x$-coordinate of a given zero, but it {\em does} alter the 
various expansions of the zeroes
in terms of $\pi_1$ and $\pi_2$ (as described in Subsection \ref{assocomp}).

In Subsection \ref{assocomp} we shall also present examples, in
the polynomial case, where zeta-zeroes are acted upon via digit 
permutations and the group $S_{(q)}$. More specifically one takes the expansion of a zero in a parameter $\pi$ and then permutes the coefficients of this
expansion in a prescribed
way to obtain other zeroes. This does NOT work for all the coefficients
but only seems to for those coefficients  which are {\em invariant} under
the change from one parameter to another! (So in this regard, we are obtaining
finer information than simply the absolute value.) This passing to the
invariants is also what gives the analogy to the use of gauge groups
in physics. 

Much more work will need to be done before we are able to place the
ideas presented here in their proper context.

Finally in \ref{classbern} we discuss how the notions of extra vanishing, etc.,
also arises in the classical theory of Bernoulli numbers. This 
builds on the famous results of von Staudt-Clausen, Adams, and Kummer.

This paper grew out of my lecture at the workshop ``Noncommutative
geometry and geometry over the field with one element'' at 
Vanderbilt University in May, 2008 as well as lectures at
subsequent conferences. It is my great pleasure to thank the
organizers of these very interesting meetings
for their kind hospitality and support.

\section{Euler's creation of $\zeta$-phenomenology}\label{discovery}
We recall here very briefly the fabulous first 
example of $\zeta$-phenomenology: Euler's numerical discovery of the
functional equation of the Riemann zeta function $\zeta (s)$.
Our treatment here follows that of \cite{ay1}; we have also
covered these ideas in \cite{go2}.

\begin{defn}\label{bernou}
The Bernoulli numbers, $B_n$, are defined by
$$\frac{x}{e^x-1}=\sum_{n=0}^\infty \frac{B_nx^n}{n!}\,.$$
\end{defn}

After many years of work, Euler computed the values
$\zeta (2n)$, $n=1,2\cdots$ in terms of Bernoulli numbers and obtained the
famous formula
\begin{equation}\label{even}
B_{2n}=(-1)^{n+1}\frac{2(2n)!}{(2\pi)^{2n}}\zeta(2n)\,.
\end{equation}

Euler
then turned his attention to the values of $\zeta (s)$ at the negative
integers where his work on special values becomes divinely
inspired! Indeed, Euler did not have the notion of analytic continuation
of complex valued functions to work with. Thus he 
relied on his instincts for beauty while working with
divergent series; nevertheless, he obtained the right values.

Euler begins with the very well known expansion
\begin{equation}\label{euler1}
\frac{1}{1-x}=1+x+x^2+x^3+\cdots+ x^n +\cdots\,.
\end{equation}
Clearly this expansion is only valid when $|x|<1$, but that does not
stop Euler. Upon putting $x=-1$, he deduces
\begin{equation}\label{euler2}
1/2=1-1+1-1+1\cdots\,,
\end{equation}
where we simply ignore questions of convergence!
He then applies the operator $x\frac{d~}{dx}$ to
Equation \ref{euler1} and again evaluates at $x=-1$ obtaining
\begin{equation}\label{euler3}
1/4=1-2+3-4+5\cdots \,.\end{equation}
Applying the operator again, Euler finds the ``trivial zero''
\begin{equation}\label{euler4}
0=1-2^2+3^2-\cdots\,,\end{equation}
and so on. Euler recognizes the sum on the right of these equations to
be the values at the negative integers of the modified $\zeta$-function
\begin{equation}\label{euler5}
\zeta^\ast(s):=(1-2^{1-s})\zeta(s)=\sum_{n=1}^\infty (-1)^{n-1}/n^s \,.
\end{equation}
The wonderful point is, of course, that these values {\em are} the values
rigorously obtained much later by Riemann. (N.B.: in \cite{ay1}, our
$\zeta^\ast(s)$ is denoted $\phi(s)$.)

Nine years later, Euler notices, at least for small $n\geq
2$, that his calculations imply 
\begin{equation}\label{euler6}
\frac{\zeta^\ast(1-n)}{\zeta^\ast(n)}= \begin{cases}\frac{(-1)^{(n/2)+1}(2^n-1)(n-1)!}{(2^{n-1}-1)\pi^n} & \text{if $n$ is even}\\
0& \text{if $n$ is odd.}\end{cases}\end{equation}
Upon rewriting Equation \ref{euler6} using his gamma function
$\Gamma (s)$ and the cosine, Euler then ``hazards'' to conjecture
\begin{equation}\label{euler7}
\frac{\zeta^\ast (1-s)}{\zeta^\ast (s)}=\frac{-\Gamma (s)(2^s-1)\cos (\pi s/2)}{
(2^{s-1}-1)\pi^s}\,,\end{equation}
which translates easily into the functional equation of $\zeta(s)$!

\begin{rem}\label{eulerrem1}
Note the important role played by the trivial zeroes in Equation
\ref{euler6} in that they
render harmless our inability to calculate explicitly $\zeta^\ast (n)$, 
or $\zeta (n)$, at odd integers $>1$.
\end{rem}

Euler then calculates both sides of Equation \ref{euler7} at
$s=1$ and obtains the same answer. To Euler, this is ``strong
justification'' for his conjecture. Of course history has proved him to
be spectacularly right! 

From now on, until \ref{classbern} where we return to classical theory,
the symbol ``$\zeta(s)$'' will be reserved for characteristic
$p$ valued functions.

\section{The factorial ideal}\label{factideal}
In order to later define Bernoulli elements in characteristic $p$, and
so explain Carlitz's von Staudt-Clausen result,  we clearly
need a good notion of ``factorial''. 
 
We begin by reviewing the basic set-up of the characteristic $p$ theory. 
We let $q=p^{n_0}$ where
$p$ is prime and $n_0$ is a positive integer. Let $X$ be a smooth
projective geometrically connected curve over the finite field
$\Fq$ with $q$-elements. Choose $\infty$ to be a fixed closed point
on $X$ of degree $d_\infty$ over $\Fq$. Thus $X-\infty$ is an affine curve and we let $A$ be the ring of
its global functions. Note that $A$ is a Dedekind domain with finite
class group and that $A^\ast=\Fq^\ast$. We let $k$ denote the quotient
field of $A$. The completion of $k$ at $\infty$ is denoted $k_\infty$ and
the completion of a fixed algebraic closure of $k_\infty$ (under the
canonical topology) is denoted ${\mathbb C}_\infty$. We let
$\F_\infty \subset k_\infty$ be the associated finite field. Set
$q_\infty:=q^{d_\infty}$ so that $\F_\infty \simeq \F_{q_\infty}$.

Of course the simplest example of such an $A$ is $\Fq[T]$,
$k=\Fq(T)$. In general
though, $A$ will not be Euclidean or factorial.

Let $x$ be a transcendental element. 

\begin{defn}\label{fact1}
1. For $i=1,2,\cdots$, we set $[i](x):=x^{q^i}-x$.\\
2. We define $L_0(x)\equiv 1$ and for $i=1,2,\cdots$, we set $L_i(x):=[i](x)[i-1](x)\cdots[1](x).$\\
3. We define $D_0(x)\equiv 1$ and for $i=1,2,\cdots$, we set 
$D_i(x):=[i](x)[i-1](x)^q\cdots[1](x)^{q^{i-1}}$.
\end{defn}

Elementary considerations of finite fields allow one to show the following
proposition (see Prop.\ 3.1.6 \cite{go1}). In it, the elements just defined
in Definition \ref{fact1} are considered as members of the polynomial ring
$\Fq[x]$.
\begin{prop}\label{easy}
{\rm 1.} $[i](x)$ is the product of all monic irreducible polynomials in $x$ whose
degree divides $i$.\\
{\rm 2.} $L_i(x)$ is the least common multiple of all polynomials in $x$ of degree
$i$.\\
{\rm 3.} $D_i(x)$ is the product of all monic polynomials in $x$ of degree
$i$. 
\end{prop}

As we will readily see later on (Proposition \ref{easy1})
the polynomials $D_i(x)$ and
$L_i(x)$ are universal for the exponential and logarithm of general
Drinfeld modules.

Our next goal is to use the functions $D_i(x)$ to define a factorial
function \`a la Carlitz. Let $j$ be an integer that we write
$q$-adically as $j=\sum_{e=0}^w c_eq^e$ where $0\leq c_e<q$ all $e$.

\begin{defn}\label{factorial}
We set
$$\Pi_j(x):=\prod_{e=0}^w D_e(x)^{c_e}\,.$$
\end{defn}
As a function of the integer $j$, $\Pi_j(x)$ satisfies many of the
same divisibility results as the classical $n!$.

Let $A$ be an arbitrary affine ring as above. We now define the basic
ideals of $A$ of interest to us. Let $f(x)\in \Fq[x]$.

\begin{defn}\label{genfact} 
We set $\tilde{f}:=\left(f(a)\right)_{a\in A}$; i.e., $\tilde{f}$ is the
ideal generated by the values of $f(x)$ on the elements of $A$.
\end{defn}

In general one would expect these ideals to be trivial (i.e., equal to
$A$ itself) as the example $f(x)=x+1$ shows. However, for the functions
given in Definition \ref{fact1}, they are highly nontrivial.

\begin{example}\label{firstone}
We show here that $\tilde{[i]}=\prod \mathfrak P$ where the product
ranges over all primes of degree (over $\Fq$) dividing $i$.
Let $\mathfrak P$ have
degree dividing $i$; then modulo $\mathfrak P$ we have $a^{q^i}=a$ (or $a^{q^i}-a=0$)
for any $a$. Thus $\mathfrak P$ must divide $\tilde{[i]}$. 
Now let $a\in A$ be a uniformizer at $\mathfrak P$. Then clearly so is
$[i](a)$. Therefore ${\mathfrak P}^2$ does not divide $\tilde{[i]}$.
Finally, a moment's thought along these
lines also shows that the only possible prime divisors of $\tilde{[i]}$ 
are those whose degree divides $i$. 
\end{example}

Let $\mathfrak P$ be a prime of $A$ with additive valuation $v_\mathfrak P$. Thakur
observed that for a function $f(x)$, $v_\mathfrak P (\tilde f)=v_\mathfrak P (\tilde
f_\mathfrak P)$ where $\tilde f_\mathfrak P$ is the analog of $\tilde f$ constructed
locally on the completion $A_\mathfrak P$ of $A$ at $\mathfrak P$. As a consequence,
we need only compute these valuations on $A=\Fq[T]$ where it is known that
the ideals associated to the functions of Definition \ref{fact1}
and Definition \ref{factorial} are generated by their values
at $x=T$. Thus, using Theorem 9.1.1 of \cite{go1},
we have the following basic factorization of $\tilde \Pi_j$.
\begin{prop}\label{sinnott}
Let $\mathfrak P$ be a prime of $A$ of degree $d$. Then 
$$v_\mathfrak P (\tilde \Pi_j)=\sum_{e\geq 1} [j/q^{ed}]\,,$$
where $[w]$ is the greatest integer function, $w\in \Q$.\end{prop}
In the fundamental case $A=\Fq[T]$, Proposition \ref{sinnott}
was proved by W. Sinnott; it is clearly a direct analog of the calculation
of the $p$-adic valuation of $n!$.

Finally, we explain the relationship with Drinfeld modules that the reader
may skip as it is not needed for the remainder of the paper. 
As before, let $k$ be the quotient field of $A$. Let $L$ be a finite
extension of $k$ with $O_L$ the ring of $A$-integers in $L$. Let $\psi$ be a
Drinfeld module of arbitrary rank over $L$ with coefficients in
$O_L$. Let $e(z)=z+\sum_{i\geq 1} e_iz^{q^i}$ and $l(z)=z+\sum_{i\geq 1} l_iz^{q^i}$
be the exponential and logarithm of the Drinfeld module (obtained say
by embedding $L$ into ${\mathbb C}_\infty$).  Let $a\in A$. 

\begin{prop}\label{easy1}
The elements $D_i(a)e_i$ and $L_i(a)l_i$ lie in $O_L$.
\end{prop}
\begin{proof}
One has the basic recurrence relations
$$e(az)=\psi_a(e(z))$$
and 
$$al(z)=l(\psi_a(z))\,.$$
The result now follows by induction and the definition of $D_i(x)$ and
$L_i(x)$\,.
\end{proof}

\section{Zeta Functions and Integral  $\zeta$-values }\label{intzvalues}
\subsection{Exponentiation of Ideals}\label{expo}
As mentioned in the introduction, we shall define these functions and  values here with
a minimum of theory and refer the reader to Chapter 8 of \cite{go1}
for the elided details. 
Our goal is to define an analog of $n^j$ where $n$ is a positive integer
and $j$ is an arbitrary integer. However, as general $A$ is {\it not}
factorial, we have to define ``$\mathfrak I^j$'' as an
element of ${\mathbb C}_\infty^\ast$ for nonprincipal
$\mathfrak I$. Here we immediately run into a notational issue in that
the symbol ``$\mathfrak I^j$'' is universally
 reserved for taking the $j$-th power
of the {\it ideal} $\mathfrak I$ in the Dedekind domain $A$. We do {\it not}
change this; rather we will use ``$\mathfrak I^{(j)}$'' for the above
element of ${\mathbb C}_\infty^\ast$ so that there will be no confusion.

Recall that the completion of $k$ at $\infty$ is denoted $k_\infty$  with
$\F_\infty\subset k_\infty$ being 
 the associated finite field; recall also that $q_\infty=q^{d_\infty}$. Fix an element $\pi\in k_\infty^\ast$ 
of order $1$. Every element $x\in k_\infty^\ast$ has a unique 
decomposition:
\begin{equation}\label{decomp}
x=\zeta_x\pi^{v_\infty (x)}u_x\,,
\end{equation}
where $\zeta_x\in \F_\infty^\ast$, $v_\infty(x)\in \Z$, and $u_x\in k_\infty$ is a $1$-unit (i.e.,
congruent to 1 modulo $(\pi)$) 
and depends on $\pi$). We say
$x$ is ``positive'' or ``monic'' if and only if $\zeta_x=1$. Clearly
the positive elements form a subgroup of finite index of $k_\infty^\ast$.

\begin{defn}\label{1unit}
We set $\langle x \rangle=\langle x \rangle_\pi:=u_x$ where $u_x$ is defined in Equation
\ref{decomp}.
\end{defn}
As mentioned, the element $\langle x \rangle$ depends on $\pi$, but
no confusion will result by not making this dependence explicit. We will have more to
say about this later.

Note that $x\mapsto \langle x \rangle$ is a homomorphism from
$k_\infty^\ast$ to its subgroup $U_1(k_\infty)$ of $1$-units.

As above, $X$ is the smooth projective curve associated to $k$. 
For any fractional ideal $I$ of $A$, we let 
$\deg_k(I)$ be the degree over $\Fq$ of the divisor
associated to $I$ on the affine curve $X-\infty$. For $\alpha\in k^\ast$,
one sets $\deg_k (\alpha)=\deg_k((\alpha))$ where $(\alpha)$ is
the associated fractional ideal; this clearly agrees with the
degree of a polynomial in $\Fq[T]$.

\begin{defn}\label{inf5}
Set ${\mathbb S}_\infty:=\C_\infty^\ast \times \Zp$\,.
\end{defn}

The space ${\mathbb S}_\infty$ plays the role of the complex numbers in our theory
in that it is the domain of ``$n^s$.'' Indeed, let
$s=(x,y)\in {\mathbb S}_\infty$ and let $\alpha\in k$ be positive. The element
$v=\langle \alpha \rangle-1$ has absolute value $<1$; thus
$\langle \alpha\rangle^y=(1+v)^y$ is easily defined and computed
via the binomial theorem.
\begin{defn}\label{inf6}
We set
\begin{equation}\label{inf7}
\alpha^s:=x^{\deg_k (\alpha)}\langle \alpha \rangle^y\,.
\end{equation} \end{defn}
\noindent
Clearly ${\mathbb S}_\infty$ is a group whose operation is written additively.
Suppose that $j\in \Z$ and $\alpha^j$ is defined in the usual sense of the
canonical $\Z$-action on the multiplicative group. Let
$\pi_\ast\in \C_\infty^\ast$ be a fixed $d_\infty$-th root of $\pi$. 
Set $s_j:=(\pi_\ast^{-j},j)\in {\mathbb S}_\infty$. One checks easily that 
Definition \ref{inf6} gives $\displaystyle \alpha^{s_j}=\alpha^j$. 
When there is no chance of confusion, we denote $s_j$ simply by ``$j$.''

In the basic case $A=\Fq[T]$ one can now proceed to define zeta-values.
However, in general $A$ has nonprincipal and positively generated ideals.
Fortunately there is a canonical and simple procedure to extend
Definition \ref{inf6} to them as follows. 
Let $\mathcal I$ be the group of fractional ideals of the Dedekind
domain $A$ and let ${\mathcal P}\subseteq \mathcal I$ be the subgroup of
principal ideals. Let ${\mathcal P}^+\subseteq \mathcal P$ be the
subgroup of principal ideals which have positive generators. It is a standard
fact that ${\mathcal I}/{\mathcal P}^+$ is a finite abelian group. The 
association 
\begin{equation}\label{inf8}
{\mathfrak h}\in {\mathcal P}^+\mapsto \langle {\mathfrak h}\rangle:= 
\langle \lambda \rangle_\pi\,,\end{equation}
 where $\lambda$ is the unique positive generator of $\mathfrak h$,
is obviously a homomorphism from ${\mathcal P}^+$ to $U_1(k_\infty)\subset \C_\infty^\ast$.

Let $U_1(\C_\infty)\subset \C_\infty^\ast$ be the group of $1$-units defined
in the obvious fashion. The binomial theorem, again, shows that
$U_1(\C_\infty)$ is a $\Zp$-module. However, it is also closed under the
unique operation of taking $p$-th roots; as such, $U_1(\C_\infty)$ is
a $\Qp$-vector space.

\begin{lemma}\label{inf9}
The mapping ${\mathcal P}^+\to U_1(\C_\infty)$ given by
$\mathfrak h\mapsto \langle {\mathfrak h}\rangle_\pi$ has a unique
extension to $\mathfrak I$ (which we also denote by $\langle ?\rangle_\pi$).
\end{lemma}
\begin{proof}
As $U_1(\C_\infty)$ is a $\Qp$-vector space, it is a divisible group; thus the
extension follows by general theory. The uniqueness then follows
by the finitude of ${\mathcal I}/{\mathcal P}^+$. \end{proof}

The next lemma explaining the dependence on $\pi$ will play a fundamental
role for us. 

\begin{lemma}\label{inf9pt5}
Let $\pi_1$ and $\pi_2$ be two positive parameters and let $I$ be a nonzero
ideal of $A$. Then
\begin{equation}\label{inf9pt6}
\langle I \rangle_{\pi_1}=(\pi_1/\pi_2)^{\deg_k(I)/d_\infty}\langle I \rangle_{\pi_2}\,.
\end{equation}\end{lemma}
\begin{proof} The formula obviously works when I is principal and positively
generated and both sides of Equation \ref{inf9pt6} are homomorphisms as functions
of $I$. Therefore,
the uniqueness of the extension from $\mathcal P^+$ to all ideals finishes
the result.\end{proof}

If $s\in {\mathbb S}_\infty$ and $I$ are chosen as above, we now set
\begin{equation}\label{inf10}
I^s:=x^{\deg_k (I)}\langle I \rangle^y\,.
\end{equation}
Thus if $\alpha\in k$ is positive one sees that $(\alpha)^s$ agrees with
$\alpha^s$ as in Equation \ref{inf7}.

For a fractional ideal $\mathfrak I$ and integer $j$, as promised, we now put
$\mathfrak I^{(j)}:=\mathfrak I^{s_j}$. Thus if $a\in k$ is positive then
$(a)^{(j)}=a^j$ by definition. The reader can check that $\mathfrak I^{(j)}$ is 
independent of $\pi$ and $\pi_\ast$ up to possible multiplication of $d_\infty$-th
root of unity.

The values ${\mathfrak I}^{(j)}$ are obviously determined multiplicatively
by $\mathfrak I^{(1)}$. Furthermore, suppose $\mathfrak I^t=(i)$ where $i$
is positive where $t$ is a positive integer (which always exists
as the ideal class group is finite) and put ${\mathfrak i}=\mathfrak I^{(1)}$. 
Then we have the basic formula
\begin{equation}\label{defI1}
\mathfrak i^t=i\,.
\end{equation}
From this it is very easy to see that the values ${\mathfrak I}^{(1)}$ generate
a finite extension $V$ of $k$ in $\C_\infty$ which is called the
{\it value field}. It is also easy to see that
$\mathfrak I$ becomes principal in this field and is generated by 
$\mathfrak I^{(1)}$.

\subsection{The $\zeta$-values}\label{zvalues}
Let $j$ be an arbitrary integer. 
\begin{defn}\label{thevalues}
We formally put  
$$\zeta(j):=\sum_{\mathfrak I}{\mathfrak I}^{(-j)}=\sum_{\mathfrak I} 
{\mathfrak I}^{s_{-j}}\,;$$ 
where $\mathfrak I$ ranges over the nonzero ideals of $A$ and $s_{-j}\in
{\mathbb S}_\infty$ was defined after Definition \ref{inf6} .\end{defn}

Because the analysis is nonArchimedean, the sum $\zeta(j)$ clearly
converges to an element of $\C_\infty$ for $j>0$. At the nonpositive
integers we must regroup the sum. More precisely, for $j\geq 0$ we write
\begin{equation}\label{regroup}
\zeta(-j)=\sum_{e=0}^\infty \left(\sum_{\deg_k (\mathfrak I)=e}
{\mathfrak I}^{(j)}\right)\,,\end{equation}
where, as above, $\deg_k (\mathfrak I)$ is the degree over $\Fq$ of
$\mathfrak I$. As is now well-known (see, e.g., Chapter 8 of \cite{go1})
for sufficiently
large $e$ the sum in parentheses vanishes. Thus the value is an algebraic
integer over $A$. 

It is known that $\zeta(-j)=0$ for $j$ a positive integer divisible by
$q_\infty-1$. These are the {\it trivial zeroes}.

\begin{defn}\label{special} 
For $j\geq 0$, we set
\begin{equation}\label{regroup1}
z(x,-j)=z_A(x,-j):=\sum_{e=0}^\infty x^{-e}\left(\sum_{\deg_k (\mathfrak I)=e}
{\mathfrak I}^{(j)}\right)\,.\end{equation}
As the sum in parentheses in Equation \ref{regroup} vanishes for sufficiently
large $e$, we see that $z(x,-j)$
is a {\em polynomial} in $x^{-1}$. These polynomials themselves also occur
as special zeta values, and, as such, are called the {\it special polynomials}. 
\end{defn}

We note that the values given above are all special values of the function given
in the next definition.

\begin{defn}\label{zetaA}
We define
\begin{equation}\label{zetaA1}
\zeta(s)=\zeta_A (s)=\zeta_{A,\pi}(s):=\sum_{\mathfrak I} {\mathfrak I}^{-s}  \end{equation}
for all $s\in {\mathbb S}_\infty$ and nonzero ideals $\mathfrak I$ of $A$.
\end{defn}
 \noindent In Lemma \ref{depend1} below  we 
will make explicit the dependence of 
$\zeta_{A,\pi}(s)$ on $\pi$.

 Let $s=(x,y)$. One then rewrites $\zeta_{A,\pi}(s)$ as 
\begin{equation}\label{rewritezeta}
\zeta_{A,\pi}(s)=\sum_{e=0}^\infty x^{-e}\left(\sum_{\deg_k (\mathfrak I)=e}
\langle{\mathfrak I}\rangle^{-y}\right)\,.\end{equation}
The analytic continuation of $\zeta_{A,\pi}(s)$ (and all such arithmetic
Dirichlet series) is accomplished by showing that these power series
are actually entire in $x^{-1}$, with very strong continuity properties
on all of ${\mathbb S}_\infty$. 

\begin{lemma}\label{depend1}
Let $\pi_1$ and $\pi_2$ be two positive parameters, $\alpha\in \mathbb C_\infty^\ast$ 
and $y_0\in \mathbb Z_p$. Then we have
\begin{equation}\label{depend2}
\zeta_{A,\pi_2}(\alpha,y_0)=\zeta_{A,\pi_1}((\pi_1/\pi_2)^{-y_0/d_\infty}\alpha,y_0)\,.\end{equation}
\end{lemma}
\begin{proof} This follows from Lemma \ref{inf9pt5} upon unraveling
the definitions.\end{proof}

Let $t=(x_0,y_0)\in {\mathbb S}_\infty$. By the  ``order of zero of $\zeta_{A,\pi}(s)$
at $t$,'' one means the order of zero of the entire power series
$\zeta_{A,\pi}(x,y_0)$ at $x=x_0$. 

\section{The Group $S_{(q)}$}\label{sp}
In this section we will introduce the automorphism groups of interest to
us. These will be subgroups of the group of homeomorphisms of $\Zp$ and 
and they will stabilize -- and so permute -- both the nonpositive and 
nonnegative integers sitting in $\Zp$. 

Let $q$ continue to be a power of $p$ and let $x\in \Zp$. Write $x$ $q$-adically 
as
\begin{equation}\label{qadic}
x=\sum_{i=0}^\infty c_i q^i
\end{equation}
where $0\leq c_i<q$ for all $i$. If $x$ is a nonnegative integer 
(so that the sum in Equation
\ref{qadic} is obviously finite), then we set
\begin{equation}\label{digitsum}
\ell_q(x)=\sum_i c_i\,.
\end{equation}

Let $\rho$ be a permutation of the set $\{0,1,2,\ldots\}$.

\begin{defn}\label{pi(n)}
We define $\rho_\ast (x)$, $x\in \Zp$, by
\begin{equation}\label{pi(n)2} 
\rho_\ast(x):=\sum_{i=0}^\infty c_i q^{\rho(i)}\,.
\end{equation}\end{defn} 
Clearly $x\mapsto \rho_\ast(x)$ gives a representation of $\rho$ as a set permutation
(in fact, as we will see in Proposition \ref{basicS}, a homeomorphism) of $\Zp$.

\begin{defn}\label{biggroup}
We let $S_{(q)}$ be the group of permutations of $\Zp$ obtained as $\rho$
varies over all permutations of $\{0,1,2,\dots\}$. \end{defn}

\begin{rem}\label{symm}
We use the notation ``$S_{(q)}$'' to avoid confusion with the symmetric
group $S_q$ on $q$-elements.\end{rem}

Note that if $q_0$ and $q_1$ are powers of $p$, and $q_0\mid q_1$, then
$S_{(q_1)}$ is naturally realized as a subgroup of $S_{(q_0)}$. 

The next proposition gives the basic properties of the mapping
$\rho_\ast(x)$.

\begin{prop}\label{basicS}
Let $\rho_\ast(x)$ be defined as above.\\
{\rm 1}. The mapping $x\mapsto \rho_\ast(x)$ is continuous on $\Zp$.\\
{\rm 2.} (``Semi-additivity'') Let $x,y,z$ be three $p$-adic integers
with $z=x+y$ and where there is no carry over of $q$-adic digits. Then
$\rho_\ast(z)=\rho_\ast(x)+\rho_\ast(y)$.\\
{\rm 3.} The mapping $x\mapsto \rho_\ast(x)$ stabilizes the nonnegative integers.\\
{\rm 4.} The mapping $x\mapsto \rho_\ast(x)$ stabilizes the negative integers.\\
{\rm 5.} Let $n$ be a nonnegative integer. Then $\ell_q(n)=\ell_q(\rho_\ast(n))$.\\
{\rm 6}. Let $n$ be an integer. Then $n\equiv \rho_\ast(n) \pmod{q-1}$.\\
\end{prop}
\begin{proof}
To see Part 1, let $j$ be a positive integer. We want to show that
the first $q^j$ expansion coefficients of $\rho_\ast(x)$ and $\rho_\ast(y)$ are the
same if $x\equiv y \pmod{q^t}$ for some positive integer $t$. Let $\phi$
be the inverse permutation to $\rho$ (as functions on the nonnegative 
integers). Choose $t$ greater than $\phi(e)$ for
$e=0,....,j-1$. Parts 2 and 3 are obvious.
To see Part 4, let $n$ be a negative integer
and let $j$ be a positive integer chosen so that $q^j+n$ is nonnegative.
Then $q$-adically we have
\begin{equation}\label{qn}
n=(q^j+n)-q^j=(q^j+n)+(q-1)q^j+(q-1)q^{j+1}+\cdots\,,
\end{equation}
as $-1=q-1+(q-1)q+(q-1)q^2\cdots$. On the other hand, $\rho_\ast(n)$ will now
clearly also have almost all of its $q$-adic coefficients equal to $q-1$ and the result
is clear.
Part 5 is clear and implies Part 6 for nonnegative $n$ as 
then we have $n\equiv \ell_q(n)\pmod{q-1}$.  Thus suppose $n$ is negative.
As in Equation \ref{qn} write $n=(q^j+n)-q^j$ with $q^j+n$ nonnegative. 
By Part 2, we have
\begin{equation}\label{qn1}
\rho_\ast(n)=\rho_\ast(q^j+n)+\rho_\ast(-q^j)\,.
\end{equation}
Clearly $\rho_\ast(-q^j)$ has almost all coefficients equal to $q-1$ with the rest
equaling $0$; thus we can write $\rho_\ast(-q^j)=m-q^t$ for some $t$ where
$m$ is positive and divisible by $q-1$. Part 6 for nonnegative integers
now implies that modulo $q-1$ we have
\begin{eqnarray*}\label{longcalc}
\rho_\ast(n)&=& \rho_\ast(q^j+n)+m-q^t\\
&\equiv& (q^j+n) -q^t\\
&\equiv&1+n-1\\
&\equiv& n\,.\end{eqnarray*}
\end{proof}
Thus, by Parts 3 and 4 of Proposition \ref{basicS}, $\rho_\ast$ permutes
both the nonpositive and nonnegative integers.

Notice further that the injection $x\mapsto p^ex$ ($e$ a positive integer)
is not in $S_{(p)}$ as it is not
surjective. However, let $n$ be a positive integer. Then clearly
$p^en=\rho_\ast(n)$ for infinitely many $\rho_\ast\in S_{(p)}$ (which may vary with $n$).
Note, however, that multiplication by $p$ will change the set
of $q$-adic digits of an integer if $q>p$, etc. Thus in this
case $pn$ will not equal $\rho_\ast(n)$ for any $\rho_\ast\in S_{(q)}$.

The reader can readily see that the cardinality of $S_{(p)}$ is 
$\mathfrak c$ (where $\mathfrak c$ is the cardinality of the continuum) as
this is the cardinality of the group of permutations of $\{0,1,2,\cdots\}$.

\begin{defn}\label{rhohat}
Let $\rho$ be as above and let $x\in \Zp$. We define
\begin{equation}\label{rhohat2}
\hat{\rho}_\ast(x):=-\rho_\ast(-x)\,.
\end{equation}
\end{defn}\noindent
It is clear that $\hat{\rho}_\ast$ also stabilizes the nonnegative and nonpositive
integers. Moreover, one can easily find examples of $\rho$ and $x$ such that
$\hat{\rho}_\ast(x)$ is {\em not} given by digit permutations.

\begin{rem}\label{decimal}
Let $x\in (0,1)$ be a real number with decimal expansion
\begin{equation}\label{decimal2}
x=.x_0x_1\ldots \,.
\end{equation}
In order to have this expansion be unique, we require that infinitely many of
the $x_i$ do not equal $9$. We can then play the same permutation game with
these decimal coefficients and we thereby obtain homeomorphisms of $(0,1)$.
Similar remarks work for an arbitrary base $b$.
\end{rem}

It is quite remarkable that the groups $S_{(q)}$ have a very natural 
relationship with binomial coefficients considered modulo $p$. This is
given in our next two results.

\begin{prop}\label{binom1}
Let $\sigma\in S_{(p)}$, $y\in \Zp$, and $k$ a nonnegative integer. Then
we have
\begin{equation}\label{binom2}
{y \choose k}\equiv  {\sigma y\choose \sigma k}\pmod{p}\,.\
\end{equation}
\end{prop}
\begin{proof} This follows immediately from Lucas' formula.\end{proof}
Next, we recall the definition of the algebra of divided power series over
a field $L$. 
These are formal sums of the form
$$\sum_{i=0}^\infty c_i\frac{z^i}{i!}\,,$$
where $\{c_i\}\subseteq L$ and one has the obvious multiplication
$$\frac{z^i}{i!}\cdot \frac{z^j}{j!}:={{i+j}\choose i}\frac{z^{i+j}}{(i+j)!}\,.$$
As the binomial coefficient is an integer, this definition works in
all characteristics.

\begin{prop}\label{binom3}
Let $i$ and $j$ be two nonnegative integers. Let $\sigma\in S_{(p)}$. Then
\begin{equation}\label{binom4}
{{i+j}\choose i}\equiv {{\sigma i+\sigma j}\choose \sigma i}\pmod{p}\,.
\end{equation}
\end{prop}
\begin{proof}
Lucas' formula shows that if there is any carry over of $p$-adic digits
in the addition of $i$ and $j$, then ${{i+j}\choose i}$ is $0$ modulo
$p$. However, there is carry over of the $p$-adic digits in the sum
of $i$ and $j$ if and only if there is carry over in the sum of
$\sigma i$ and $\sigma j$; in this case both sums are $0$ modulo $p$.
If there is no carry over, then the result follows from Part 2 of Proposition
\ref{basicS}\end{proof} 

\begin{cor}\label{binom5}
Let $\sigma$ be as in the proposition. 
Then the mapping $\frac{z^i}{i!}\mapsto \frac{z^{\sigma i}}{\sigma (i)!}$ is
an algebra automorphism of the divided power series in characteristic $p$.
\end{cor}

As was explained in Section 8.22 of \cite{go1} and \cite{co1}, the 
algebras of measures in characteristic $p$ are isomorphic to divided power
series algebras. More precisely, let $R:=\Fq[[u]]$. One first picks
a basis for the Banach space of continuous $\Fq$-linear functions; then,
using the $q$-adic expansion of an integer $t$, one obtains an associated
basis for the Banach algebra of {\it all} continuous functions from $R$ to
itself. One then sees readily that the algebra of $R$-valued measures on
$R$ (equipped with convolution as usual) is thus isomorphic to the formal
divided power series algebra over $R$. Therefore the next corollary
follows immediately.

\begin{cor}\label{binom6} The group $S_{(p)}$ acts as automorphisms
of the convolution algebra of $R$-valued measures on $R$.\end{cor}

\begin{example}\label{binom7} As an example of how Corollary \ref{binom6}
may be used, consider the field $k_\infty(\pi_\ast)\simeq \F_{q_\infty}((\pi_\ast))$.
It contains the ring $\F_{q_\infty}[[\pi_\ast]]$ which is obviously the completion
of $\F_{q_\infty}[\pi_\ast]$ at the ideal generated by $\pi_\ast$. One can then,
e.g., use the Carlitz polynomial basis, as Section 8.22 of \cite{go1}, to
obtain the isomorphism with the divided power series algebra.
\end{example}

\begin{rem}\label{toSinfty}
At first glance, classical theory would indicate that there should be
an extension of the action of $S_{(p)}$ to all of ${\mathbb S}_\infty$ which is analytic
in the first variable and which takes integer powers to integer powers.
However, one knows that the bijective rigid analytic maps from 
${\mathbb G}_m$ to itself are of the $x\mapsto cx^{\pm 1}$ for some nonzero
constant $c$. There simply do not appear to be enough
of these functions to extend the action
of  $S_{(p)}$ on the integers inside $\mathbb S_\infty$ to all of $\mathbb S_\infty$ {\em analytically} (in the first variable).  
\end{rem}

The above remark leads us to look at non-analytic maps in hopes of extending the 
action of
$S_{(q)}$ to $\mathbb C_\infty^*$ and thus all of $\mathbb S_\infty$.   Calculations, to be
given in \ref{assocomp},  lead to the following construction. 

\begin{defn}\label{k1}
We let $K_1:=\mathbb F_\infty((\pi_\ast))$.
\end{defn}
\noindent
Obviously $K_1$ is totally ramified extension of $K$. Let $x\in K_1^*$ which we
write as $x=\sum\limits_{i\gg-\infty} c_i\pi_*^i$. 

\begin{defn}\label{k2}  Let $\rho$ be a permutation of $\{0,1,2,\dots\}$ as
before and let $x\in K_1^\ast$ be written as just above. We set
\begin{equation}\label{k3}
\rho_*(x):=\sum_{i\gg-\infty} c_i\pi_\ast^{\rho_*(i)}\in K_1^\ast\,.
\end{equation}\end{defn}\noindent
This definition precisely works because, by Proposition \ref{basicS}, we know
that $\rho_\ast$ stabilizes both the nonpositive and 
nonnegative integers. Moreover, one can
easily modify the proof of Proposition \ref{basicS} to deduce the continuity of
$\rho_\ast$ on $K_1^\ast$. 

\begin{defn}\label{k4} We set 
$$\mathbb S_{\infty,\pi_\ast}:=K_1^\ast\times \mathbb Z_p\subset \mathbb S_\infty\,.$$
For $(x,y)\in \mathbb S_{\infty,\pi_\ast}$, we set
\begin{equation}\label{k5}
\rho_\ast (x,y):=(\rho_\ast(x),\hat{\rho}_\ast(y))\in \mathbb S_{\infty,\pi_\ast}\,,
\end{equation}
with $\hat{\rho}$ as in Definition \ref{rhohat}.
\end{defn}

One then easily computes that for an integer $j$
\begin{equation}\label{k6}
\rho_\ast(s_{-j})=s_{-\rho_\ast(j)}\,,
\end{equation}
which we would like for any extension of the action of $\rho$ on $\Zp$ and is
very natural in terms of trivial zeroes (see Subsection \ref{negev}). As of
this writing, we know of no natural way to extend Definition \ref{k4} to
all of $\mathbb S_\infty$.

\begin{rem}\label{vadicinter}
Let $v$ be a prime of $A$ where $A$ is now arbitrary. We discuss here
briefly the $v$-adic theory associated to $\zeta_{A,\pi}(s)$.
Let $k_v$ be the
associated completion of $k$ with fixed algebraic closure $\bar{k}_v$ equipped
with the canonical absolute value etc. Let
$\C_v$ be the associated completion.
As explained in Section 8.3 of \cite{go1}, the special polynomials
(Definition \ref{special}) interpolate to two-variable functions on
$$\C_v^\ast\times {\mathbb S}_v$$
where ${\mathbb S}_v=\Zp\times \Z/(q^\delta-1)$, and $\delta$ is determined
by the degree of $v$ plus a choice of injection, over $k$,
of the value field $V$ into $\C_v$. (This construction is quite similar
to the $p$-adic interpolation of classical $L$-values.)
The appropriate group of homeomorphisms
of ${\mathbb S}_v$ is then $S_{(q^\delta)}$ as its elements preserve
congruence classes modulo $q^\delta-1$. 
\end{rem}

Finally, we finish this section with the following finiteness result
whose easy proof will be left to the reader. Note first that by Part 5 of
Proposition \ref{basicS}, for any constant $c$,  the set $X(q,c)$ consisting
of all positive integers $n$ with $\ell_q(n)=c$ is stable under $S_{(q)}$.

\begin{prop}\label{finiteness1} 
The set $X(q,c)$ consists of finitely many orbits of $S_{(q)}$.
\end{prop}

\section{$S_{(q)}$ as symmetries of $\zeta(s)$}\label{zetasymm}
In this last section, we present the evidence showing how $S_{(q)}$ and its
subgroups arise as symmetries of the $\zeta$-functions and
values of Section \ref{zvalues}.
The evidence we have is Eulerian by its very nature as we {\it only}
use the special values. However, unlike Euler, we cannot now guess at the
mechanism that exhibits these groups as automorphisms of the full two-variable
zeta function.

As we saw in Proposition \ref{basicS}, the group $S_{(q)}$ permutes
both the positive and negative integers. Each set of integers separately
gives evidence that $S_{(q)}$ acts as symmetries of $\zeta(s)$.
However, it may ultimately turn out that both types of
evidence are really manifestations of the same underlying symmetries.

We begin with the evidence from the negative integers.
\subsection{Evidence from $\zeta$-values at negative integers}\label{negev}
In this section we present the evidence that $S_{(q)}$ acts as symmetries
of $\zeta (s)$ arising from the negative integers. We believe that this evidence
has greater impact than the evidence given in Subsection \ref{posev} (at the
positive integers) because it represents
actual symmetries associated to the zeroes of $\zeta(s)$.
We shall see, experimentally at least, that
the orders of vanishing of $\zeta(s)$ at negative
integers (as recalled in Subsection \ref{zvalues})
appear to be invariants of the action of $S_{(q)}$.

Here is what is known about such vanishing in general.
Recall that $q_\infty:=q^{d_\infty}$ where $d_\infty$ is the degree of
$\infty$ over $\Fq$. Let $j$ be
a positive integer which is divisible by $q_\infty-1$. 
Then it is known that $\zeta(-j)=0$ (see, e.g. Section 8.13 of
\cite{go1}). Current theory
naturally gives very classical looking {\it lower bounds} on the order
of vanishing of these ``trivial zeroes.'' 
In the case $d_\infty=1$ this bound is $1$; when $d_\infty>1$, it may
be greater than $1$. As our examples here all have $d_\infty=1$, we refer
the interested reader to \cite{go1} for the general case.

The first example is $A=\Fq[T]$. Here it is known  \cite{sh1} that {\it all}
zeroes are simple and that $\zeta(-j)\neq 0$ for $j\not\equiv 0 \pmod{q-1}$.
By Part 6 of Proposition \ref{basicS}, we have $j\equiv \rho_\ast(j)\pmod{q-1}$.
Thus the next proposition follows immediately.
\begin{prop}\label{invnegint}
Let $A=\Fq[T]$. Then the order of vanishing of $\zeta(s)$ at $-j$, $j$ positive,
is an invariant of the action of $S_{(q)}$ on the positive integers. \end{prop}

By itself, Proposition \ref{invnegint} is certainly not overwhelming evidence
for realizing $S_{(q)}$ as a symmetry of the full zeta function.
However, in the case of some nonpolynomial $A$ with $\infty$ rational, 
Dinesh Thakur \cite{th1}, Theorem 5.4.9 of \cite{th2}, and Javier Diaz-Vargas
\cite{dv2} have produced some fundamentally important calculations on zeroes at
negative integers. They found examples of trivial zeroes where the 
natural lower bound was {\it not} the exact order of vanishing;
of course, this is something that never
happens in the classical analytic theory of $L$-series (but see
\ref{classbern} for analogies in the classical 
theory of Bernoulli numbers as well
as justification for the terminology which follows).
Zeroes arising at the negative integers with order of vanishing 
strictly greater than the predicted
lower bound will called ``irregular,'' otherwise they are said to be
``regular.''  (N.B.: In earlier versions of this work,
as well as other papers, we called such zeroes ``nonclassical''.)

The calculations of Thakur and Diaz-Vargas seem to imply that the irregular
trivial zeroes which occur at $-j$, $j$ a nonnegative integer, actually
occur where $\ell_q(j)$ is {\it bounded}. 
Moreover, their calculations {\it also} continue to exhibit $S_{(q)}$ invariance
of the orders of vanishing of even these irregular zeroes.

As such we expect Proposition \ref{invnegint} to remain true for general
$A$ where $q$ will need to be replaced by $q_\infty$.

\begin{rem}\label{finiteness2}
If Proposition \ref{invnegint} is true for general $A$, and the
irregular trivial zeroes are characterized by having bounded sums of
$q$-adic digits, then Proposition
\ref{finiteness1} immediately implies that there are only finitely many
possibilities for the order of zero at irregular trivial zeroes. This result
is quite reasonable to believe.
\end{rem}

As analytic objects on ${\mathbb S}_\infty$, 
our zeta functions are naturally 1-parameter
families of entire power series where the parameter is $y\in \Zp$. Having
such a huge group acting on $\Zp$ may ultimately give us good control
of the family. We will see serious evidence for this in our next
subsection.

\subsection{Evidence from special polynomials}\label{specialev}
We begin by recalling some relevant history as well as results. Let $i$
and $e$ be nonnegative integers. Let $A=\Fq[T]$ and let $A^+(e)$ be the monic
polynomials of degree $e$ inside $A$. Define 
\begin{equation}\label{sum1}
S_e(i):=\sum_{f\in A^+(e)}f^i\,.
\end{equation}  
Let $\pi$ be a uniformizer in $k_\infty$ as before with associated
$1$-unit parts $\langle f \rangle=\pi^ef$. Let $t\in \Zp$. Define
\begin{equation}\label{sum2}
\tilde{S}_e(t):=\sum_{f\in A^+(e)}\langle f \rangle^t\,.
\end{equation}
Clearly $\tilde{S}_e(t)$ is continuous in $t$.
Obviously, for $i$ a nonnegative integer, $\pi^{ei}S_e(i)=\tilde{S}_e(i)$,
and so both sides are nonzero for the same $i$. 

In \cite{ca4}, L.\ Carlitz mentions the following necessary and sufficient
criterion for $S_e(i)$ to be nonzero: There should be an expression
\begin{equation}\label{decomp1}
i=i_0+i_1+\cdots+i_e\,,\end{equation}
such that all $i_j$ are nonnegative, and such that
\begin{enumerate}
\item there is no carryover of $p$-adic digits in the sum,
\item for  $0\leq j<e$, we have $i_j>0$ and divisible by $q-1$.
\end{enumerate}
\noindent
We agree to call such an expression for $i$ an ``admissible representation
for $i$ relative to $e$.'' The necessity actually follows easily
from expanding $f \in A^+(e)$ by the multinomial theorem.

Dinesh Thakur astutely 
realized that Carlitz's criterion, along with a corresponding
formula for the degree of $S_e(i)$, could be used to compute the
Newton polygons associated to the power series arising from $\zeta_A$ (which do
{\em not} depend on the choice of $\pi$).
However, it fell to J.\ Sheats, in \cite{sh1}, to give the first rigorous
proof of Carlitz's assertions for general $q$. Along the way, Sheats also
established the following results.

\begin{prop}\label{sheats1}
{\rm 1}. $\tilde{S}_e(t)=0$ implies that $t$ is a nonnegative integer.\\
{\rm 2}. Let $i$ be a nonnegative integer. Then 
$$S_e(i)=0\Rightarrow S_{e+1}(i)=0\,,$$
for all $e\geq 0$.
\end{prop}

\begin{prop}\label{admissible}
Let $i$ be a nonnegative integer with $i=i_0+i_1+\cdots+i_e$ an
admissible representation of $i$ relative to $e$. Let $\rho_\ast\in
S_{(q)}$. Then $\rho_\ast (i)=\rho_\ast(i_0)+\cdots+\rho_\ast (i_e)$ is an
admissible representation of $\rho_\ast(i)$ relative to $e$.
\end{prop}
\begin{proof}
This follows immediately from Parts 2 and 6 of Proposition \ref{basicS}.
\end{proof}

\begin{cor}\label{admpoint5}
Let $e$, $i$ and $\rho$ be as above. Then we have
\begin{equation}\label{admpoint6}
S_e(i)=0 \Longleftrightarrow S_e(\rho_*(i))=0\,.
\end{equation}
\end{cor}
\begin{proof} This follows immediately from the Proposition and
Sheats' results. \end{proof}

\begin{cor}\label{admissible1}
Let $j$ be a nonnegative integer with associated special polynomial
$z(x,-j)$ ({\rm Definition \ref{special}}). Let $\rho_\ast$ be as in the proposition. Then $z(x,-j)$ and 
$z(x,-\rho_\ast(j))$ have the same degree in $x^{-1}$.\end{cor}

\begin{cor}\label{admissible2} 
Let $X_e\subset \Zp$ be the zero set of $\tilde{S}_e(t)$. Then
$X_e$ is stable under the action of $S_{(q)}$ on $t$.\end{cor}
\begin{proof} By Sheats we know that any zeroes of $\tilde{S}_e(t)$ must
be nonnegative integers. By the proposition, we therefore see that
$$\tilde{S}_e(t)\neq 0 \Longleftrightarrow \tilde{S}_e(\rho_\ast (t))\neq 0$$ 
for all $t$ and homeomorphisms $\rho_\ast$. 
So the result follows immediately. \end{proof}

\begin{rem}\label{admissible3}
Calculations suggest that $X_e$ consists of finitely many orbits of
$S_{(q)}$. \end{rem}
\begin{rem}\label{admissible4} 
Warren Sinnott \cite{si1} 
has studied a large class of functions on $\Zp$ containing
uniform limits of $\tilde{S}_e(t)$ (for arbitrary $e$ of course). Let
$f(t)$ be one nontrivial such function. Then Sinnott shows that the zero set of
$f(t)$ cannot contain an open set, unlike general continuous functions on
$\Zp$. The results given above suggest that the zero set of $f(t)$ may
in fact be countable.
\end{rem}

\begin{rem}\label{admissible5}
Gebhard B\"ockle has informed me that Sheats' results imply that the
degree of $z(x,-j)$ as a polynomial in $x^{-1}$ equals
$$\min_i\{[\ell_q(jp^i)/(q-1)]\}$$
where $i=0,1,\ldots, n_0-1$ and $q=p^{n_0}$. \end{rem}

\begin{rem}\label{adm6}
The results of Jeff Sheats just mentioned were established along the way of
Sheats' proof of the Riemann hypothesis for $\zeta_{\Fq[T]}(s)$. That is, as
mentioned before, 
following D. Wan, D. Thakur, J. Diaz-Vargas and B. Poonen, J.\ Sheats was able to
establish that $\zeta_{\Fq[T]}(s)$ has {\em at most} one 
(including multiplicity)
zero in $x^{-1}$ of a given absolute value for each fixed $y\in \Zp$; this
implies that the the zeroes in $x^{-1}$ must be simple
and ``lie on the line'' $\Fq((1/T))$.\end{rem}

It seems reasonable to expect that generalizations of the above results 
should hold for arbitrary $A$, etc. 

Let $K$ be a local non-Archimedean field of characteristic $0$. A standard
argument using Krasner's Lemma implies that $K$ has only finitely many
extensions of bounded degree. The same argument works in finite characteristic
for degrees less than the characteristic (so one avoids inseparability
issues).  

Therefore, suppose that $\mathfrak C$ is an orbit of the nonnegative
integers under $S_{(q_\infty)}$. Suppose that the special
polynomials associated to elements in $\mathfrak C$ all have the same degree which we
assume is less than $p$. We can
conclude that their zeroes all belong to a finite extension of $k_\infty$.

\begin{question}\label{bigq}
Is it true in general that the zeroes of special polynomials $i$ and
$j$, where $i$ and $j$ are in the same orbit under $S_{(q_\infty)}$, lie
in a finite extension of $k_\infty$?
\end{question}
The results of Sheats for $A=\Fq[T]$ imply that all the zeroes of special
polynomials lie
in $k_\infty$ in this case and so here the answer is obviously yes. 
Moreover, in \ref{assocomp} we will exhibit some evidence that the zeroes of special
polynomials in the same orbit are related via $S_{(q_\infty)}$, so it makes
sense to wonder if this relationship carries over to their splitting fields...

Finally, let $A=\Fq[T]$ again, and for the rest of this Subsection. Let $v=(f)$ 
be a prime of degree $d$ in $A$ with $f$ monic. Recall
that in Remark \ref{vadicinter} we explained briefly how our functions
interpolated $v$-adically; in the case at hand, it is easy to
see that $\delta=d$. Let $A^+(e)_v$ be the monic polynomials of
degree $e$ which are prime to $v$. Define
\begin{equation}\label{vadicsum}
\tilde{S}_{e,v}(t):=\sum_{f\in A_v^+(e)} f^t\,,
\end{equation}
where $t$ now belongs to $\mathbb S_v=\Zp\times \Z/(q^d-1)$. 
It is also reasonable to expect our results just given to hold $v$-adically. 
Let $X_{e,v}\subseteq {\mathbb S}_v$ be the zero set of $\tilde{S}_{e,v}(t)$
We simply point out that we already can deduce many orbits of $S_{(q^d)}$ on 
the nonnegative integers which lie in $X_{e,v}$. Indeed, if $e<d$ then $A_v^+(e)=A^+(e)$;
so $X_e\subseteq X_{e,v}$ (recall $X_e$ consists of nonnegative integers which
also project densely into $\mathbb S_v$) and is obviously stable under
$S_{(q^d)}$. 

Suppose now that $e\geq d$. Let $i$ be a nonnegative integer. 
Clearly,
$$\tilde{S}_{e,v}(i)= S_e(i)-f^iS_{e-d}(i)\,.$$
Thus, by Part 2 of Proposition \ref{sheats1} we see that $X_{e-d}\subseteq 
X_{e,v}$ and is obviously stable under $S_{(q^d)}$.

\subsubsection{Associated Computations} \label{assocomp}
In this subsection $A$ will always be $\Fq[T]$. We shall present here
some calculations, as well as a related result,
that indicate rather strongly that the zeroes
of $\zeta_A(s)$ are ``stable'' under the action given in Definition \ref{k4}.
Indeed, in Remark \ref{adm6}, we recalled that Sheats has established that all 
that all the zeroes of $\zeta_{\Fq[T]}(s)$ have $x$-coordinate
in $\Fq((1/T))$ {\em and} that 
there is at most
one zero of any given absolute value for fixed $y\in \Zp$; 
so, in particular, as elements of $\mathbb S_\infty$, the zeroes
actually lie in $\mathbb S_{\infty,\pi_\ast}=\mathbb S_{\infty,\pi}$.

We put the word stable, above, in quotes because the situation is far from clear at the moment.
Indeed, what actually appears to be stable is that part of the zeroes that
is {\em invariant} upon changing one positive parameter for another as
indicated in Lemma \ref{depend1}. However, we do 
not yet have a clear general statement of this even for $\Fq[T]$. Moreover,
for more general $A$, one knows that there are zeroes of 
$\zeta_A(s)$ lying {\em outside} $\mathbb S_{\infty,\pi_\ast}$.

\begin{prop}\label{Uone}
The set of positive parameters is a principal homogeneous space for the group
$U_1(K)$ of $1$-units.\end{prop}

\begin{proof}
Let $\pi_1$ and $\pi_2$ be two positive parameters. Then the quotient
$\pi_1/\pi_2$ obviously belongs to $U_1(K)$. The result follows. \end{proof}

\begin{rem}\label{Utwo}
Let $\pi$ be a positive parameter and let $u\in K$ be a unit $\not \equiv 1\pmod{\pi}$.
Then $u\pi$ will be a positive parameter for a different choice of ``positive''.
\end{rem}

\begin{rem}\label{Uthree} In quantum field theory one is interested in the wave
function $\psi$. This is a complex valued function whose absolute value $\vert \psi
\vert$ gives predictions for results in physical experiments. In other words,
only $\vert \psi\vert$ has physical meaning. Let $U:=\{e^{i\theta}\}$ be the
``gauge group'' of complex numbers of norm $1$; obviously multiplication by an 
element of $U$ does {\em not} change the absolute value (and conversely).  
One is thus free to multiply
$\psi$ be functions of the form $e^{ip(x,t)}$ for arbitrary real valued
functions $p(x,t)$ depending on position and time. From this freedom many
remarkable things may be deduced. In the function field case, the ``gauge group''
$U_1(K)$ acts on the parameters and the values in line with Lemma \ref{depend1}. However,
as we shall see, we are interested in more invariants than just given
by the absolute value!
\end{rem}

Let $\pi_1$ and $\pi_2$ be two positive uniformizers. Lemma \ref{depend1} tells
that if  $(\alpha,y_0)$ is a zero of $\zeta_{A,\pi_1}(s)$  then 
$((\pi_1/\pi_2)^{y_0}\alpha,y_0)$ is a zero of $\zeta_{A,\pi_2}(s)$. It is then reasonable
to ask what invariants arise in passing from $\alpha$ to $u^y\alpha$ for $u=\pi_1/\pi_2\in U_1(K)$.
For instance, obviously $\vert\alpha\vert_\infty=\vert u^y\alpha\vert_\infty$, where
$\vert ?\vert_\infty$ is the absolute value associated to the place $\infty$, (and which is
in keeping with the analogy to quantum theory just given). Let $t\in \Z$ be chosen
so that $\vert\alpha/\pi_1^t\vert_\infty=\vert u^y\alpha/\pi^t_2\vert_\infty=1$. 
Thus we can write
$$\alpha=a\pi_1^t +\text{\{lower~terms\}}$$
and 
$$u^y\alpha=b\pi_2^t+\text{\{lower~terms\}}$$
where $a$ and $b$ are nonzero elements of $\Fq$; as $u$ is a principal unit, we
conclude $a=b$. In other words, the whole monomial of least degree is an invariant
of our action. We shall see shortly that there may be other invariant
monomials.

We now present some elementary calculations that appear to indicate
an action of $S_{(q)}$ on zeta-zeroes.

\begin{example}\label{firstq3} Let Let $A=\mathbb F_3 [T]$, $\pi = 1/T$
and $i=13=1+3+3^2$.  Let $H_{13}\subset S_{(3)}$ be the isotropy
subgroup of $13$.  Let $\sigma_\ast\in H_{13}$; note that
$\sigma_\ast$ {\em must} permute the first three digits of a $3$-adic
integer (otherwise $\sigma_\ast (13)\not= 13$); this is simple but is also
the key observation.  Hand calculations then give
\begin{equation}\label{secondq3}
\zeta_{A,\pi} (x,-13)=1-(\pi^4+\pi^{10}+\pi^{12})x^{-1}\,.
\end{equation}
Note that
\begin{eqnarray*}\label{thirdq3}
\sigma_\ast (\pi^{4}+\pi^{10}+\pi^{12}) &= &\pi^{\sigma_\ast (4)} +
\pi^{\sigma_\ast (10)} +\pi^{\sigma_\ast (12)}\\
&= &\pi^4 +\pi^{10} +\pi^{12}
\end{eqnarray*}
as $\sigma_\ast (4)=\sigma_\ast (1+3) = 1+3$, $1+9$, or $3+9$, etc., since the
first three digits of $13$ are permuted! Thus 
\begin{equation}\label{fourthq3}\sigma_\ast
(\pi^4+\pi^{10}+\pi^{12},-13)=(\pi^4+\pi^{10}+\pi^{12},-13)\,,
\end{equation}
and the zero is obviously fixed under the action arising from digit
permutations as given in Definition \ref{k4}.
\end{example}

\begin{example}\label{firstq2}
Let $A=\mathbb F_2[T]$, $\pi = 1/T$
and $i=3=1+2$ or $i=5=1+4=1+2^2$.  For $i=3$ one finds
\begin{equation}\label{secondq2}
\zeta_{A,\pi} (x,-3) = (1-\alpha x^{-1})(1-\beta x^{-1})
\end{equation}
where $\alpha =\pi^3$ (this is the trivial zero, which will arise for all
$i$ as $q-1=1$ in this case)
and $\beta = \pi
+\pi^2$.  As above, if $\sigma_\ast\in H_3=$ the isotropy subgroup of
$3$, then $\sigma_\ast (\alpha ) =\alpha$ and $\sigma_\ast (\beta)=\beta$ so 
that $\sigma_\ast$ fixes both $(\alpha ,-3)$ and $(\beta
,-3)$.  A similar calculation holds for $i=5$ where there is the trivial
zero $(\pi^5,-5)$ and another zero $(\pi+\pi^4,-5)$.  Note also that there
exists permutations $\rho$ such that $\rho_\ast (3)=5$.  One then finds
immediately that
\begin{equation}\label{thirdq2}
\rho_\ast (\pi^3,-3)=(\pi^5,-5)
\end{equation}
and
\begin{equation}\label{fourthq2}
\rho_\ast (\pi +\pi^2,-3)=(\pi +\pi^4,-5)\,.
\end{equation}
\end{example}

\begin{rem}\label{constraints}
The reader should note that having an action of $S_{(q)}$ on the zeroes
as in the above examples appears to put extremely strong constraints 
on the form that the zeroes may take. For instance, let $q=2$ and
$i=3$ as above. Note that if the second zero was $(\pi+\pi^4,-3)$ we
would run into great difficulty as $\pi^4$ has an infinite orbit
under $H_3$.\end{rem}

In the arithmetic of function fields, the truly canonical ideas do {\em not}
depend on the choice of positive parameter $\pi$. For instance,
let $q=2$ and $i=3$ as in Example \ref{firstq2}. The reader can easily
find examples of positive parameters $\pi$ so that when we transform
the zero $(\pi+\pi^2,-3)$ using Lemma \ref{depend1}, we obtain a zero
with expansion terms involving higher powers of $\pi$. As we just
remarked, this causes absolute havoc with our hoped for action of $S_{(2)}$.
{\em However}, the answer to this appears to lie in the use of our
``gauge group'' $U_1(K)$ as in our next example.

\begin{example}\label{gauge1}
Let $q=2$ and now put $\pi_0=1/T$.  As in Example \ref{firstq2}
we see that $\zeta_{A,1/T}(x,-3)=(1-\alpha x^{-1})(1-\beta x^{-1})$
with $\alpha = \pi^3_0$ and $\beta = \pi_0+\pi^2_0$.  Let $\pi$ be {\em any}
other uniformizer and let $\tilde{\alpha}$, $\tilde{\beta}$ be the
transformed zeroes under Lemma \ref{depend1}. 
It follows immediately that  $\widetilde{\pi}^3_0=\pi^3$ and
\begin{eqnarray*}\label{gauge2}
\widetilde{\pi_0+\pi^2_0} = (\pi/\pi_0)^3(\pi_0+\pi^2_0) &= &\pi^3
(\pi^{-1}_0 +\pi^{-2}_0)\\
&= &\pi^3(T+1/T^2)\,.
\end{eqnarray*}
Now $\pi_0=\pi +a_2\pi^2+a_3\pi^3+\cdots$ where $a_i\in \mathbb F_2$
and may be completely arbitrary (as $\pi$ is a completely arbitrary positive
uniformizer).  Using long division to compute
$T=\pi^{-1}_0$, we find that
\begin{equation}\label{gauge3}
\widetilde{\pi_0+\pi^2_0} = \pi +\pi^2 +0\cdot \pi^3 +O(\pi^4)\,.
\end{equation}
In other words, the terms ``$\pi +\pi^2 +0\cdot \pi^3$'' are invariant of the 
choice of $\pi$, {\em and} these are the ones that transform correctly under our action of $S_{(2)}$!\end{example}

One can readily find other such examples. Finally we present
a rather general result that gives further support for believing that
the zeroes of $\zeta_{\Fq[T]}(s)$ are acted upon by $S_{(q)}$ in line
with the above examples.

Let $q=p^{n_0}$ as above and let $y\in \mathbb Z_p$ be written as
\begin{equation}\label{collapse1}
y=\sum_{i=0}^N c_iq^{e_i}\,,
\end{equation}
where $0\neq c_i<q$ for all $i$, $N$ may be infinite, and $e_0<e_1<\cdots$.

\begin{defn}\label{collapse2}
We set
\begin{equation}\label{collapse3}
y_c:=\sum_{i=0}^Nc_iq^i\,,
\end{equation}
and call it the {\it $q$-adic collapse of $y$}.\end{defn}
\noindent
For instance the $2$-adic collapse of $5$ is $3$.

Let $j$ now be a positive integer with associated $q$-adic collapse
$j_c$. Let $\rho$ be a permutation of $\{0,1,\dots\}$ which has,
in the above notation, the
special values $\rho(i)=e_i$, for $i=0,\dots, N$. (Obviously,
as $j$ is a positive integer, $N<\infty$.)

By Corollary \ref{admissible1}, we see that $\zeta(x,-j_c)$ and 
$\zeta(x,-j)$ have the same degree $d$ in $x^{-1}$. From Sheats we know that
the zeroes lie in $K$ and that there is at most one zero of each polynomial
of a given absolute value.
Let $\{\alpha_1,\dots ,\alpha_d\}$ be the zeroes of $\zeta (x,-j_c)$,
and let $\{\beta_1,\dots ,\beta_d\}$ be the zeroes of $\zeta (x,-j)$,
where
$$
\ord_\pi \alpha_1 <\ord_\pi \alpha_2 <\cdots < 
\ord_\pi \alpha_d
$$
and
$$
\ord_\pi \beta_1 <\cdots <\ord_\pi \beta_d\,.
$$

\begin{prop}\label{collapse4}
Let $j$ be as above, but
where we assume that all $q$-adic digits of $j$ are $<p$, (for
instance, when $q=p$).  Then
\begin{equation}\label{collapse5}
\ord_\pi (\beta_i)=\rho_\ast (\ord_\pi \alpha_i)
\end{equation}
for $i=1,\dots ,d$.\end{prop}
\begin{proof}
This follows from the calculations of these orders given by Diaz-Vargas
in \cite{dv1}\end{proof}

Calculations seem to indicate that the result is true for arbitrary
$y\in \Zp$. We hope to return to these themes in future work.

\subsection{Evidence from $\zeta$-values at positive integers}\label{posev}
In this section we discuss the evidence from the positive  integers
and, in particular, the evidence arising from Carlitz's analog,
Theorem \ref{car1}, 
of the classical von Staudt-Clausen result (which computes
the denominators of Bernoulli numbers and is recalled in \ref{classbern} when
we return to classical number theory).

The work of David Hayes on ``sign-normalized'' rank one Drinfeld modules
(see \cite{hay1} or Chapter 7 of \cite{go1}) 
shows the existence of a special Drinfeld module $\psi$ with the
following properties: It is defined over the ring of integers in
a certain Hilbert Class Field of $k$ (ramified at $\infty$) lying in
$\C_\infty$ which we denote by $H^+$. Let $\mathfrak I$ be an ideal of
$A$. Then the product of all $\mathfrak I$-division
values of $\psi$ lies in $H^+$ and is an explicit generator of ${\mathcal O}^+
\mathfrak I$ where $\mathcal O^+$ is the ring of $A$-integers. The lattice $L$
associated to $\psi$ may be written $A\xi$ for a transcendental element
$\xi\in \mathbb C_\infty$. 

Let $T:=V\cdot H^+$ be the compositum of $V$ and $H^+$ where $V$ is defined
after Equation \ref{defI1}. The following
result is shown in \cite{go1} (Theorem 8.18.3).

\begin{theorem}\label{poszval}
Let $j$ be a positive integer divisible by $q_\infty-1$ and let $\zeta(j)$ be
defined as in {\rm Definition \ref{thevalues}}. Then
\begin{equation}\label{incomp}
0\neq \zeta(j)/\xi^j\in T\,.
\end{equation}
\end{theorem}

Theorem \ref{poszval} was established in the basic $\Fq[T]$-case by
L. Carlitz in the 1930's; in this case, both $V$ and $H^+$ equal $k$. 

Let ${\mathcal O}_T$ be the $A$-integers of $T$. 
\begin{defn}\label{genbcnums}
Let $j$ be divisible by $q_\infty-1$. We define $\widetilde{BC}_j$ 
to be the ${\mathcal O}_T$  fractional
ideal generated by $\tilde{\Pi}_j\zeta(j)/\xi^j$.
\end{defn} 
We call the fractional ideal $\widetilde{BC}_j$ the ``$j$-th 
Bernoulli-Carlitz fractional ideal'' as, again, these were originally
defined (as elements in $\Fq(T)$) by Carlitz in the 1930's and are clearly
analogous to the classical Bernoulli numbers (Definition \ref{bernou}).

\begin{defn}\label{Adenom}
Let ${\mathfrak d}_j:=\{a\in A\mid a\widetilde{BC}_j\subseteq {\mathcal O}_T\}$.
\end{defn}
We call ${\mathfrak d}_j$ the ``$A$-denominator of $\widetilde{BC}_j$;'' it is
obviously an ideal of $A$.

For $A=\Fq[T]$, Carlitz (\cite{ca1}, \cite{ca2}, \cite{ca3}) 
gives an explicit calculation of
${\mathfrak d}_j$ which we now recall. Let $q=p^{n_0}>2$ for the moment.
\begin{theorem}\label{car1}
{\rm (Carlitz)} There are two conditions on $j$:\\
{\rm 1}. $h:=\ell_p(j)/(p-1)n_0$ is integral.\\
{\rm 2}. $q^h-1$ divides $j$.\\
If $j$ satisfies both conditions, then ${\mathfrak d}_j$ is the product
of all prime ideals of degree $h$. If $j$ does not satisfy both conditions,
then ${\mathfrak d}_j=(1)$.\end{theorem}

Carlitz's result gives us the first (historically) 
indication of an action of $S_{(q)}$
on $\zeta$-values. This is given in the following corollary.

\begin{cor}\label{invar1}
Let $\mathfrak P$ be a prime ideal of $A$ of degree $h$ with additive
valuation $v_\mathfrak P$. Then
$v_{\mathfrak P}({\mathfrak d}_j)$ is an invariant of the
action of $S_{(q^h)}$ on the positive integers divisible by $q-1$,
\end{cor}

What about $q=2$. Here (\cite{ca3}) the same result holds if $h\neq 2$. More
precisely Carlitz established the following result.
\begin{theorem}\label{car2}
{\rm (Carlitz)} Let $q=2$ and consider the system given in {\rm Theorem
\ref{car1}}. If this system is consistent for $h=\ell_2(j)\neq 2$,  
then ${\mathfrak d}_j$ is the product of all prime ideals of degree
$h$. If it is consistent for $h=2$, then for $j$ even we
have 
\begin{equation}\label{q2}
{\mathfrak d}_j=(T^2+T+1)\,,
\end{equation}
while, for $j$ odd, we have
\begin{equation}\label{q21}
{\mathfrak d}_j=(T^2+T\cdot T^2+T+1)\,.
\end{equation}
If the system is inconsistent and $j$ is of the form $2^\alpha+1$
(so $\ell_2(j)=2$), then
\begin{equation}\label{q22}
{\mathfrak d}_j=(T^2+T)\,.
\end{equation}
If it is inconsistent and $j$ cannot be written as $2^\alpha+1$, then
${\mathfrak d}_j=(1)$.\end{theorem}
\begin{cor}\label{car3} ($q=2$)
Let $\mathfrak P$ be a prime of $A$ of degree $h$ and suppose that
first that $h\neq 1$. Then ${v}_{\mathfrak P}(\mathfrak d_j)$
is an invariant of the action of $S_{(2^h)}$ on the positive integers.
If $h=1$ then ${v}_{\mathfrak P}(\mathfrak d_j)$ is an invariant
of the subgroup $\tilde{S}_{(4)}$ of $S_{(4)}$ arising from permutations 
of $\{0,1,2,\ldots\}$ fixing $0$.
\end{cor}

It is reasonable to expect these symmetries to persist when Carlitz's
results are generalized to arbitrary $A$ where, again, one will need to
replace $q$ with $q_\infty$.

Let $A=\Fq[T]$ and let $\mathfrak P$ be a prime of $A$ of degree
$h$. For simplicity assume that $q>2$. Let $i$ and $j$ be
two positive integers which are in the same orbit of $S_{(q^h)}$ and
are divisible by $q-1$. Suppose $i$ satisfies Carlitz's two
conditions given in Theorem \ref{car1}; then by Corollary \ref{invar1}
so does $j$. In this case we see that $\widetilde{BC}_i$ and 
$\widetilde{BC}_j$ have the same order, $-1$, at $\mathfrak P$.

\begin{question}\label{bernques}
Let $\mathfrak P$ be a prime of $A=\Fq[T]$ of degree $h$.
Suppose $i$ and $j$ are two nonnegative integers which are divisible
by $q-1$ and in the same orbit of $S_{(q^h)}$. Do 
$\widetilde{BC}_i$ and $\widetilde{BC}_j$ have the same order at $\mathfrak P$?
\end{question}

One can formulate variants of Question \ref{bernques} for arbitrary
$A$. We discuss a classical variant in our next subsection.

\subsubsection{Classical Bernoulli numbers}\label{classbern}
In this final subsection we formulate
an analog of Question \ref{bernques}
for classical Bernoulli numbers (Section \ref{discovery}) as well
as present a reasonable solution to it as an exercise in $p$-adic
continuity.

Let $p$ now be an
odd prime and let $B_n$ be the classical Bernoulli number for a positive
integer $n$, so that the functional equation for the Riemann zeta function
gives
\begin{equation}\label{specriem}
-B_n/n=\zeta(1-n)\,,
\end{equation}
where $\zeta(s)$ now is the Riemann zeta function as in Section \ref{discovery}.

Next we recall some fundamental, and famous, classical results on which
we shall build. First, we have
the classical {\it Adams congruences} which state that if $p-1$ does
not divide $n$, then $B_n/n$ is integral at $p$. The following result
now follows immediately.

\begin{prop}\label{adamscor}
Let $n$ be a positive integer which is even and not divisible by $p-1$.
Let $v_p$ denote the additive valuation associated to $p$. Then
\begin{equation}\label{adams1}
v_p(B_n)\geq v_p(n)\,.
\end{equation}\end{prop}
Let $t:=v_p(n)$. If $t>0$ then
say that ``$B_n$ has a trivial zero of order at least
$t$  at $p$.'' In general, if $t=v_p(B_n)$ , then 
we say that $n$ is ``regular''  with respect to
$p$; otherwise we say it is ``irregular.''

Note obviously that $n$ is regular with respect to $p$ if and only
if $v_p(B_n/n)=0$; thus we see the connection with the usual notion
of ``regularity'' from the theory of cyclotomic fields. This simple
analogy is the motivation for our use of this terminology also for zeta zeroes
at negative integers (Subsection \ref{negev}).

Next suppose that $p-1$ divides $n$, $n$ even. As above, let $\mathfrak d_n$
be the denominator of $B_n$. In this case the classical 
{\it von Staudt-Clausen Theorem} implies that $v_p({\mathfrak d}_n)=1$
and otherwise vanishes (N.B.: $v_2({\mathfrak d}_n)$ is identically
$1$). We say that
``$B_n$ has a simple pole at $p$.''

Finally, we recall the powerful {\it Kummer congruences} which state
the following: Let $i>0$ be a positive integer which is not divisible
by $p-1$. Let $j>0$ be another integer and assume that
\begin{equation}\label{kummer1}
i\equiv j\pmod{p^{b-1}(p-1)}\,,
\end{equation}
then
\begin{equation}\label{kummer2}
(1-p^{i-1}) B_i/i \equiv (1-p^{j-1}) B_j/j \pmod {p^b}\,.
 \end{equation}

Suppose now that $\rho_\ast\in S_{(p)}$ and let $m:=\rho_*(n)$ where $n$ is an
even integer. We have seen that $m\equiv n \pmod{p-1}$ and, in
particular, is also even. Question \ref{bernques} leads us to consider
whether $v_p(B_n)=v_p(B_m)$ in general. In this formulation, 
the answer is ``no''
as one sees with $p=5$, $n=2$ and $m=10=2\cdot5$. Indeed, 
the numerator of $B_2$ is $1$ and the numerator of
$B_{10}=5$ (so both $2$ and $10$ are regular with respect to $p=5$
and $B_{10}$ has a trivial zero at $5$).

Note, however,  that the classical von Staudt-Clausen result is simpler than its
function field counterpart and only depends upon divisibility of $n$ by
$p-1$. As such,  we suppress the action of $S_{(p)}$ here and concentrate
on the more elementary notion of congruence.
Under this assumption, a reasonable analogy to Question \ref{bernques}
may then be given as follows.

\begin{question} \label{bernques2}
Let $n$ be a positive even integer. Does there exist a integer $t\geq 0$,
which depends on $n$, such that if $m$ is another positive integer with
$m\equiv n \pmod{(p-1)p^t}$, then $v_p(B_n)=v_p(B_m)$?
\end{question}
The following result provides a positive answer.

\begin{prop}\label{bernques3} Let $p$ be an odd prime
and $n$ a positive even integer. 
Let 
$$t_0:=\max\{v_p(n),v_p(B_n/n)\}$$
and $t=t_0+1$. Suppose that
$m$ is another positive integer with $m\equiv n\pmod{(p-1)p^t}$. 
Then $v_p(B_n)=v_p(B_m)$.
\end{prop}

\begin{proof} One sees immediately that $v_p(n)=v_p(m)$. Moreover,
the Kummer Congruences tell us that $v_p(B_n/n)=v_p(B_m/m)$. The result
follows.\end{proof} 

It may be interesting to find an upper bound for the integer $t$ given
in the proposition which depends only on $n$. Moreover, I do not
know how often an integer $m$ satisfying the hypothesis of Proposition
\ref{bernques2} can be obtained from $n$ by permuting its $p$-adic
digits.

It is my pleasure to acknowledge very useful input from Warren
Sinnott and Bernd Kellner.

\end{document}